\documentclass[10pt]{amsart}
\usepackage{amsmath}
\usepackage{amsthm}
\usepackage{amsfonts}
\usepackage{amssymb}
\usepackage{latexsym}
\usepackage[mathscr]{eucal}

\usepackage{hyperref}

\DeclareMathAlphabet\mathoo{U}{eur}{b}{n}

\theoremstyle{plain}
\newtheorem{theorem}{Theorem}
\newtheorem{proposition}[theorem]{Proposition}
\newtheorem{corollary}[theorem]{Corollary}

\theoremstyle{remark}

\begin{document}

\title[Computation of a function of a matrix]{Computation of a function of a matrix \\with close eigenvalues by means of\\
the Newton interpolating polynomial}

\author{V.G. Kurbatov}
\address{Russian Presidential Academy of National Economy and Public
Administration} \curraddr{} \email{kv51@inbox.ru}

\author{I.V. Kurbatova}\thanks{The second author was supported by the Russian Foundation
for Basic Research, research project No. 13-01-00378.}
\address{Air Force Academy of the Ministry of Defense of the Russian
Federation} \curraddr{} \email{la\_soleil@bk.ru}

\date{}

\begin{abstract}
An algorithm for computing an analytic function of a matrix $A$ is described. The
algorithm is intended for the case where $A$ has some close eigenvalues, and clusters
(subsets) of close eigenvalues are separated from each other. This algorithm is a
modification of some well known and widely used algorithms. A novel feature is an
approximate calculation of divided differences for the Newton interpolating polynomial in
a special way. This modification does not require to reorder the Schur triangular form
and to solve Sylvester equations.
\end{abstract}

\maketitle

 \thispagestyle{empty}
\section{Introduction}\label{s:introduction}

Matrix functions (see, e.g.,~\cite{Gantmakher,Golub-Van Loan96}) play a role of a useful
language and an effective tool in many applications. The most popular matrix function is
the matrix exponential; it is closely connected with solutions of differential equations.
One of the problems (see~\cite{Moler-Van Loan78,Moler-Van Loan03}) that arise in the
process of calculating a function $f$ of a matrix $A$ is appearing of expressions of the
form $\frac{f(\mu_1)-f(\mu_2)}{\mu_1-\mu_2}$, where $\mu_1$ and $\mu_2$ are eigenvalues
of the matrix~$A$. If the eigenvalues $\mu_1$ and $\mu_2$ are close to each other, the
literal meaning of the expression $\frac{f(\mu_1)-f(\mu_2)}{\mu_1-\mu_2}$ implies the
calculation of differences of close numbers, which leads to essential loss of accuracy,
see~\cite{Moler-Van Loan78,Moler-Van Loan03} for a detailed discussion. If the difference
$\mu_1-\mu_2$ is very small, it is reasonable to change approximately the expression
$\frac{f(\mu_1)-f(\mu_2)}{\mu_1-\mu_2}$ by $f'(\mu_1)$. But if the difference
$\mu_1-\mu_2$ is neither large nor small, then the problem becomes more serious.

A way of overcoming this problem was discussed in \cite{Parlett74,Parlett76,
Kagstrom77,McCurdy-Ng-Parlett,Parlett-Ng,Davies-Higham2003,Higham-Al-Mohy2010}. The
initial step consists in the transformation of the matrix $A$ to a triangular form by
means of the Schur algorithm (see, e.g.,~\cite[ch.~7]{Golub-Van Loan96}). As a result, in
particular, the spectrum of the matrix $A$ becomes known. After that the spectrum is
divided into clusters (parts) $S_j$ in such a way that the eigenvalues within a cluster
are close to each other, and the eigenvalues from different clusters lie far apart, for a
detailed discussion of this procedure we refer
to~\cite{McCurdy-Ng-Parlett,Davies-Higham2003}. Then, the Schur triangular representation
\emph{is reordered}, i.e., it is changed so that the eigenvalues (which are the diagonal
elements of the triangular matrix) from the same cluster are situated near to each other
(the standard Schur algorithm does not guarantee such an ordering even in the case of the
real spectrum). Thus, one arrives at the block triangular representation in which the
spectra of different diagonal blocks are concentrated in small sets (clusters), and, at
the same time, are widely spaced from each other. Finally, the function of the block
triangular matrix is calculated recursively, i.e., one block diagonal after the other,
see~\cite{Parlett74,Parlett76}; these computations employ solving the Sylvester
equations. When the function $f$ is applied to an individual diagonal block, the function
$f$ is replaced by its Taylor expansion, see~\cite{Kagstrom77,Davies-Higham2003}. As a
consequence the problem connected with the calculation of divided differences of the kind
$\frac{f(\mu_1)-f(\mu_2)}{\mu_1-\mu_2}$ with close $\mu_1$ and $\mu_2$ disappears.

In this article a modification of the above algorithm is offered. It allows one to avoid
the procedures of reordering the triangular Schur representation and solving the
Sylvester equations.

The idea of the algorithm consists in the calculation of the approximate Newton
interpolating polynomial $p(A)$ of $A$ instead of $f(A)$, where the points of
interpolation are eigenvalues $\mu_i$ of $A$. In this case the problem of cancellation of
close numbers in the divided differences $\frac{f(\mu_1)-f(\mu_2)}{\mu_1-\mu_2}$ moves to
the stage of forming the Newton interpolating polynomial $p$. The problem is solved in
the old way, i.e., by means of an approximation of $f$ by its Taylor polynomial in a
neighbourhood of close eigenvalues, which leads to a calculation of $p$ with high
accuracy. So, it remains to substitute the matrix $A$ into~$p$. Numerical experiments
show that the algorithm can be used when the order $n$ of the matrix $A$ is less then~30.
The algorithm implies that the calculation of a matrix polynomial is a solvable problem.
In this connection we refer to~\cite{Van Loan79,Higham-Knight1995} where the calculation
of matrix powers and matrix polynomials are discussed.

In Section~\ref{s:Newton polynomial} the definition of the Newton interpolating
polynomial is recalled. In Section~\ref{s:partial Newton polynomial} the calculation of
divided differences is analysed. The whole algorithm is described in
Section~\ref{s:algorithm}. Some numerical experiments are presented in Section~\ref{s:num
exper}. In Section~\ref{s:pulse char} an application to a symbolic calculation of the
impulse response of a dynamical system is discussed.

\section{The Newton interpolating polynomial}\label{s:Newton polynomial}

Let $\mu_1$, $\mu_2$, \dots, $\mu_n$ be given complex numbers (some of them may coincide
with others) called \emph{points of interpolation}. Let a complex-valued function $f$ be
defined and analytic in a neighbourhood of these points. \emph{Divided differences}
of the function $f$ with respect to the points $\mu_1$,
$\mu_2$, \dots, $\mu_n$ are defined (see, e.g.,~\cite{Gelfond}) by the recurrent
relations
 \begin{equation}\label{e:divided differences}
 \begin{split}
\Delta_i^{i}f&=f^{[0]}(\mu_i)=f(\mu_i),\\
\Delta_i^{i+1}f&=f^{[1]}(\mu_i,\mu_{i+1})=\frac{f^{[0]}(\mu_{i+1})-f^{[0]}(\mu_i)}{\mu_{i+1}-\mu_i},\\
\Delta_i^{i+m}f&=f^{[m]}(\mu_i,\dots,\mu_{i+m})=\frac{f^{[m-1]}(\mu_{i+1},\dots,\mu_{i+m})
-f^{[m-1]}(\mu_{i},\dots,\mu_{i+m-1})} {\mu_{i+m}-\mu_i}.
 \end{split}
 \end{equation}
In these formulas, if the denominator vanishes, then the quotient means the derivative
with respect to one of the arguments of the previous divided difference (this agreement
may by derived by continuity from Corollary~\ref{c:Gelfond}).

 \begin{proposition}\label{p:Gelfond}
The divided differences possess the representation
\begin{equation*}
f^{[m]}(\mu_i,\mu_{i+1},\dots,\mu_{i+m})=\frac1{2\pi
i}\int_\Gamma\frac{f(\lambda)\,d\lambda}{(\lambda-\mu_{i})(\lambda-\mu_{i+1})\dots(\lambda-\mu_{i+m})},
\end{equation*}
where the contour $\Gamma$ encloses all the points of interpolation $\mu_i$, $\mu_{i+1}$,
\dots, $\mu_{i+m}$.
 \end{proposition}
 \begin{proof}
See~\cite[ch.~1, \S~4.3, formula (54)]{Gelfond}.
 \end{proof}

 \begin{corollary}\label{c:Gelfond}
The divided difference $\Delta_i^{i+m}f=f^{[m]}(\mu_i,\mu_{i+1},\dots,\mu_{i+m})$ is a
symmetric function, i.e., it does not depend on the order of its arguments $\mu_i$,
$\mu_{i+1}$, \dots, $\mu_{i+m}$.
 \end{corollary}
 \begin{proof}
The assertion follows from Proposition~\ref{p:Gelfond}.
 \end{proof}

It is convenient to arrange the divided differences into the triangular table
 \begin{equation}\label{e:table}
\begin{array}{llllllll}
\Delta_1^{1}&\Delta_2^{2}&\dots&\Delta_k^{k}&\dots&\Delta_{n-2}^{n-2}&\Delta_{n-1}^{n-1}&\Delta_n^{n}\\
\Delta_1^{2}&\Delta_2^{3}&\dots&\Delta_k^{k+1}&\dots&\Delta_{n-2}^{n-1}&\Delta_{n-1}^{n}\\
\Delta_1^{3}&\Delta_2^{4}&\dots&\Delta_k^{k+2}&\dots&\Delta_{n-2}^{n}\\
\hdotsfor[2]{5}\\
\Delta_1^{n-1}&\Delta_2^{n}\\
\Delta_1^{n}
\end{array}
 \end{equation}

\emph{The interpolating polynomial in the Newton form} or shortly \emph{the Newton
interpolating polynomial} with respect to the points $\mu_1$, $\mu_2$, \dots, $\mu_n$ is
(see, e.g.,~\cite{Gelfond}) the polynomial
\begin{equation}\label{e:Newton polynomial}
 \begin{split}
p(\lambda)&=\Delta_1^{1}f+(\lambda-\mu_1)\Delta_1^{2}f+
(\lambda-\mu_2)(\lambda-\mu_1)\Delta_1^{3}f\\
&+(\lambda-\mu_3)(\lambda-\mu_2)(\lambda-\mu_1)\Delta_1^{4}f+\dots\\
&+(\lambda-\mu_{n-1})(\lambda-\mu_{n-2})\dots(\lambda-\mu_1)\Delta_1^{n}f.
 \end{split}
\end{equation}
We stress that the Newton interpolating polynomial contains only divided differences from
the first column of table~\eqref{e:table}. The main property of the interpolating
polynomial is (see, e.g.,~\cite{Gelfond}) the equalities
\begin{equation*}
p(\mu_i)=f(\mu_i),\qquad i=1,2,\dots,n.
\end{equation*}
If two or more points $\mu_1$, $\mu_2$, \dots, $\mu_n$ coincide, the last formula is
understood as the equality of the corresponding derivatives.

A discussion of the direct application of the Newton interpolating polynomial to the
calculation of matrix functions can be found
in~\cite{McCurdy-Ng-Parlett,Parlett-Ng,Dehghan-Hajarian09}.

\section{The principal divided differences}\label{s:partial Newton polynomial}
Let us discuss the structure of the Newton interpolating polynomial in a special case.
Let us assume that some $k$ points $\mu_{l+1}$, $\mu_{l+2}$, \dots, $\mu_{l+k}$ are
situated close to each other (in particular, some of them may coincide with one another),
and the rest of points $\mu_{1}$, $\mu_{2}$, \dots, $\mu_{l}$ and $\mu_{l+k+1}$,
$\mu_{l+k+2}$, \dots, $\mu_n$ are situated far apart from them (it is possible that some
of them are also close to each other). For better distinguishing, sometimes we denote the
points from the second group by the symbols  $\nu_{1}$, $\nu_{2}$, \dots, $\nu_{l}$ and
$\nu_{l+k+1}$, $\nu_{l+k+2}$, \dots, $\nu_n$ instead of  $\mu_{1}$, $\mu_{2}$, \dots,
$\mu_{l}$ and $\mu_{l+k+1}$, $\mu_{l+k+2}$, \dots, $\mu_n$.

We call a divided difference $\Delta_i^{i+m}g$ \emph{principal} if its indices satisfy
the inequalities $l+1\le i$ and $i+m\le l+k$. In table~\eqref{e:table} principal divided
differences form a triangle with the legs of the length~$k$. For the sake of clarity we
write out some principal divided differences, for example, when $k=3$:
\begin{align*}
&f(\mu_{l+1}),&&f(\mu_{l+2}),&&f(\mu_{l+3}),\\
&\frac{f(\mu_{l+2})-f(\mu_{l+1})}{\mu_{l+2}-\mu_{l+1}},&&\frac{f(\mu_{l+3})-f(\mu_{l+2})}{\mu_{l+3}-\mu_{l+2}},\\
&\frac{\frac{f(\mu_{l+3})-f(\mu_{l+2})}{\mu_{l+3}-\mu_{l+2}}
-\frac{f(\mu_{l+2})-f(\mu_{l+1})}{\mu_{l+2}-\mu_{l+1}}}{\mu_{l+3}-\mu_{l+1}}.
\end{align*}
We call a divided difference $\Delta_i^{i+m}g$ \emph{non-principal} if its indices
satisfy the inequalities $l+1\le i\le l+k$ and $i+m>l+k$. In table~\eqref{e:table}
non-principle divided differences are situated below principle ones.

According to definition~\eqref{e:divided differences}, divided differences must be
computed row by row (in any order in each row). We note that principle divided
differences may be computed separately (row by row), i.e., independently of other divided
differences.

We stress that problems with division by a difference of close numbers arise when we
compute principal divided differences, but does not arise when we compute non-principle
ones. Actually, according to~\eqref{e:divided differences}, when we compute the
denominator of a principal divided difference $\Delta_{l+i}^{l+m}f$ the new difference
$\mu_{l+m}-\mu_{l+i}$ (of close numbers $\mu_{l+m}$ and $\mu_i$) arise. But when we
compute non-principal divided differences, new denominators may have only the form
$\nu_j-\mu_{l+i}$ (which are not small numbers).

We summarize these observations in the following proposition.

 \begin{proposition}\label{p:Newton projector}
The divided differences of the function $f$ possess the following properties.
 \begin{itemize}
 \item [\rm(a)] Principal divided differences
depend only on the points $\mu_{l+1}$, $\mu_{l+2}$, \dots, $\mu_{l+k}$ and the values of
the function $f$ {\rm(}and may be its derivatives{\rm)} at these points.
 \item [\rm(b)] The denominators of the form $\mu_{l+m}-\mu_{l+i}$ with $i,m=1,\dots,k$
{\rm(}which are small numbers{\rm)} may appear only during the calculation of principal
divided differences {\rm(}after that they may get into non-principal divided differences,
but only in an implicit form as parts of already calculated principle ones{\rm)}.
 \item [\rm(c)] The denominators of the form $\nu_j-\mu_{l+i}$
with $i=1,\dots,k$ and $j\neq l+1,\dots,l+k$ {\rm(}which are large numbers{\rm)} may
appear only in non-principal divided differences.
 \end{itemize}
 \end{proposition}
 \begin{proof} The complete proof is by induction on $m$.
 \end{proof}

Proposition~\ref{p:Newton projector}(c) shows that if we meet with success in organizing
the calculation of principal divided differences without essential losses of accuracy,
then there will be no large losses of accuracy in the calculation of non-principal ones
as well.

In order to compute the principal divided differences we approximate the function $f$ in
a neighbourhood of the points $\mu_{l+1}$, $\mu_{l+2}$, \dots, $\mu_{l+k}$ by a
polynomial~$h$. The simplest and universal way is to take for $h$ the Taylor polynomial
of the function $f$ about the middle point
\begin{equation*}
\bar\mu=\frac{\mu_{l+1}+\mu_{l+2}+\dots+\mu_{l+k}}{k}.
\end{equation*}
It is desirable that some disk centered in $\bar\mu$ contains all the points $\mu_{l+1}$,
$\mu_{l+2}$, \dots, $\mu_{l+k}$ and lies in the domain of the function~$f$. We note that
it would be more convenient to take for $\bar\mu$ the centre of a disk of the smallest
radius that contains all the points $\mu_{l+1}$, $\mu_{l+2}$, \dots, $\mu_{l+k}$. But
finding such a centre requires additional efforts.

In the absence of a complementary information on the matrix $A$ it is reasonable to take
the degree of $h$ not less than $k-1$. In exact arithmetic, the higher degree of $h$, the
better approximation may be achieved, but the enlargement of the degree of $h$ slows down
the calculations. Thus a compromise is necessary, for a detailed discussion
see~\cite{Davies-Higham2003}.

So, let us assume that in a neighbourhood of $\bar\mu$ the function $f$ is replaced by
the polynomial
\begin{equation}\label{e:h}
h(\lambda)=\sum_{\alpha=0}^{k+\gamma}c_\alpha(\lambda-\bar\mu)^\alpha
\end{equation}
of degree $k+\gamma$, where $\gamma=-1,0,1,\dots$. For the sake of clarity, we write out
principal divided differences of $h$ for the special case when $k=4$ and $\gamma=0$:
\begin{align*}
\Delta_{l+1}^{l+1}h&=c_0+c_1\xi_{l+1}+c_2\xi_{l+1}^2+c_3\xi_{l+1}^3+c_4\xi_{l+1}^4,\\
\Delta_{l+1}^{l+2}h&=c_1+c_2(\xi_{l+1}+\xi_{l+2})
+c_3(\xi_{l+1}^2+\xi_{l+1}\xi_{l+2}+\xi_{l+2}^2)\\
&+c_4(\xi_{l+1}^3+\xi_{l+1}^2\xi_{l+2}+\xi_{l+1}\xi_{l+2}^2+\xi_{l+2}^3),\\
\Delta_{l+1}^{l+3}h&=c_2+c_3(\xi_{l+1}+\xi_{l+2}+\xi_{l+3})\\
&+c_4(\xi_{l+1}^2+\xi_{l+2}^2+\xi_{l+2}^2+\xi_{l+3}^2+\xi_{l+1}\xi_{l+2}+\xi_{l+1}\xi_{l+3}+\xi_{l+2}\xi_{l+3}),\\
\Delta_{l+1}^{l+4}h&=c_3+c_4(\xi_{l+1}+\xi_{l+2}+\xi_{l+3}+\xi_{l+4}),
\end{align*}
where $\xi_i$ is the shorthand for $\mu_i-\bar\mu$.

The structure of $h$ is described in the following proposition.

 \begin{proposition}\label{p:Newton projector2}
Let a function $h$ be a polynomial of the form~\eqref{e:h} in a neighbourhood of the
points $\mu_i$, $\mu_{i+1}$, \dots, $\mu_{i+m}$. Then the divided differences of the
function $h$ possess the representation
 \begin{equation}\label{e:Delta h}
 \begin{split}
\Delta_i^{i+m}h&=\sum_{\alpha=m}^{k+\gamma}c_{\alpha}\sigma_{\alpha-m}(\xi_i,\xi_{i+1},\dots,\xi_{i+m})\\
&=c_{m}+\sum_{\alpha=m+1}^{k+\gamma}c_{\alpha}\sigma_{\alpha-m}(\xi_i,\xi_{i+1},\dots,\xi_{i+m}),
 \end{split}
\end{equation}
where the homogeneous polynomials $\sigma_\alpha$ are defined by the formulas
\begin{align*}
\sigma_0(\xi_i,\xi_{i+1},\dots,\xi_{i+m})&=1,\\
\sigma_\alpha(\xi_i,\xi_{i+1},\dots,\xi_{i+m})&=\sum_{i_0+i_1+\dots+i_m=\alpha}\xi_i^{i_0}\xi_{i+1}^{i_1}\dots\xi_{i+m}^{i_m}.
\end{align*}
In particular, the divided differences of the function $h$ does not contain differences
of close numbers in denominators {\rm(}in this representation{\rm)}.
 \end{proposition}
 \begin{proof}
Since we are interested only in principal divided differences, by
Proposition~\ref{p:Newton projector}(a), we may assume that $k=n$. We proceed by
induction on $m$. For $m=0$ the assertion is evident. We assume that
representation~\eqref{e:Delta h} holds for
$\Delta_i^{i+m-1}h=h^{[m-1]}(\xi_i,\xi_{i+1},\dots,\xi_{i+m-1})$. We show that
representation~\eqref{e:Delta h} holds for
$\Delta_i^{i+m}h=h^{[m]}(\xi_i,\xi_{i+1},\dots,\xi_{i+m-1},\xi_{i+m})$. We begin with the
auxiliary identity (if $\xi_{i+m}=\xi_i$, the division by $\xi_{i+m}-\xi_i$ is understood
as the differentiation; cf. the definition of a divided difference)
\begin{align*}
&\frac{\sigma_\alpha(\xi_{i+1},\dots,\xi_{i+m-1},\xi_{i+m})
-\sigma_\alpha(\xi_{i+1},\dots,\xi_{i+m-1},\xi_{i})} {\xi_{i+m}-\xi_i}\\
&=\sum_{i_0+i_1+\dots+i_{m-1}=\alpha}\xi_{i+1}^{i_1}\dots\xi_{i+m-1}^{i_m}
\frac{\xi_{i+m}^{i_0}-\xi_i^{i_0}}{\xi_{i+m}-\xi_i}\\
&=\sum_{i_0+i_1+\dots+i_{m-1}=\alpha}\xi_{i+1}^{i_1}\dots\xi_{i+m-1}^{i_m}
\sum_{i_{i+m}=0}^{i_0-1}\xi_i^{i_0-i_{i+m}}\xi_{i+m}^{i_{i+m}}\\
&=\sum_{i_0+i_1+\dots+i_{m}=\alpha-1}\xi_i^{i_0}\xi_{i+1}^{i_1}\dots\xi_{i+m-1}^{i_m}\xi_{i+m}^{i_{i+m}}\\
&=\sigma_{\alpha-1}(\xi_i,\xi_{i+1},\dots,\xi_{i+m}).
\end{align*}

By definition and by Corollary~\ref{c:Gelfond} we can represent $\Delta_i^{i+m}h$ in the
form
\begin{equation*}
\Delta_i^{i+m}h=\frac{h^{[m-1]}(\xi_{i+1},\dots,\xi_{i+m-1},\xi_{i+m})
-h^{[m-1]}(\xi_{i+1},\dots,\xi_{i+m-1},\xi_{i})} {\xi_{i+m}-\xi_i}.
\end{equation*}
Therefore (by the above auxiliary identity)
\begin{align*}
\Delta_i^{i+m}h&=\frac{h^{[m-1]}(\xi_{i+1},\dots,\xi_{i+m-1},\xi_{i+m})
-h^{[m-1]}(\xi_{i+1},\dots,\xi_{i+m-1},\xi_{i})} {\xi_{i+m}-\xi_i}\\
&=\frac1{\xi_{i+m}-\xi_i}\Biggl(\sum\limits_{\alpha=m-1}^{k+l}c_{\alpha}\sigma_{\alpha-m+1}(\xi_{i+1},\dots,\xi_{i+m-1},\xi_{i+m})\\
&-\sum\limits_{\alpha=m-1}^{k+l}c_{\alpha}\sigma_{\alpha-m+1}(\xi_{i+1},\dots,\xi_{i+m-1},\xi_{i})\Biggr)
\\
&=\sum_{\alpha=m}^{k+l}c_{\alpha}\sigma_{\alpha-m}(\xi_i,\xi_{i+1},\dots,\xi_{i+m})\\
&=c_{m}+\sum_{\alpha=m+1}^{k+l}c_{\alpha}\sigma_{\alpha-m}(\xi_i,\xi_{i+1},\dots,\xi_{i+m}).
\end{align*}
Thus, representation~\eqref{e:Delta h} is established.
 \end{proof}

Finally, we arrive at the following theorem.
 \begin{theorem}\label{t:Newton polynomial of h}
Let in a neighbourhood of the points $\mu_{l+1}$, $\mu_{l+2}$, \dots, $\mu_{l+k}$ the
function $f$ coincide with a polynomial. Then the divided differences
$\Delta_{i}^{i+m}f$, $l+1\le i\le l+k$, can be computed without subtraction of close
numbers in denominators.
 \end{theorem}
 \begin{proof}
For principal divided differences the proof follows from Proposition~\ref{p:Newton
projector2}. For non-principal divided differences the proof follows from
Proposition~\ref{p:Newton projector}.
 \end{proof}

\section{An algorithm for the calculation of a matrix function}\label{s:algorithm}
Let $A$ be a square matrix of the size $n\times n$. In order to find its eigenvalues we
apply to the matrix $A$ the Schur algorithm, see~\cite{Golub-Van Loan96}. (If the matrix
$A$ is real, the real form of the Schur algorithm can be used; it generates a block
triangular matrix with diagonal elements of the sizes $2\times2$ and $1\times1$.) We
write out the eigenvalues of $A$ (counted with multiplicity): $\mu_1$, $\mu_2$, \dots,
$\mu_n$.

Let a small number $\delta>0$ be given. We split the set of all eigenvalues into clusters
(parts) $S_1$, $S_2$, \dots, $S_\beta$ in such a way that~\cite{Davies-Higham2003}
 \begin{itemize}
 \item [\rm(a)] $|\mu_i-\mu_j|\ge\delta$ for any $\mu_i\in S_i$ and
$\mu_j\in S_j$ with $i\neq j$;
 \item [\rm(b)] for any pair $\mu_1,\mu_\omega$ from the same cluster $S_j$ there exists a chain
$\mu_1,\dots,\mu_k=\mu_\omega\in S_i$ such that $|\mu_i-\mu_{i+1}|<\delta$ for
$i=1,\dots,k-1$.
 \end{itemize}
An algorithm for splitting eigenvalues into clusters can be found
in~\cite[p.~474]{Davies-Higham2003}. We denote by $k_j$ the number of eigenvalues in the
cluster~$S_j$.

We reorder the eigenvalues. Namely, we arrange the eigenvalues in such a way that the
eigenvalues from the same cluster are situated one after the other. Note that this
reordering is essentially more simple than the reordering of diagonal elements in the
Schur triangular representation.

Let $f$ be an analytic function defined in a neighbourhood of the spectrum of the matrix
$A$. Our aim is an approximate calculation of the matrix $f(A)$. We recall (see,
e.g.,~\cite[theorem 11.2.1]{Golub-Van Loan96}) that if at all points $\mu_i$ of the
spectrum of $A$ the values of functions $f$ and $p$ and their derivatives are close to
each other up to the order of the multiplicity of $\mu_i$, then $p(A)\approx f(A)$. So,
for an approximation to $f(A)$ we take $p(A)$, where a polynomial $p$ approximates $f$ in
a neighbourhood of the spectrum of~$A$.

According to Section~\ref{s:partial Newton polynomial}, on each cluster $S_j$ we
approximate the function $f$ by a polynomial
\begin{equation}\label{e:h_j}
h_j(\lambda)=\sum_{\alpha=0}^{k_j+\gamma_j}c_{\alpha j}(\lambda-\bar\mu_j)^\alpha,
\end{equation}
where $\gamma_j$ is chosen so that the difference $f-h_j$ is small in a neighbourhood of
the cluster~$S_j$. For each $j$, we compute principal divided differences of $f$ as
principal divided differences of $h_j$ in accordance with formula~\eqref{e:Delta h}.
After principal divided differences are calculated for \emph{all} $j$, we compute
non-principal divided differences of $f$ by definition~\eqref{e:divided differences}.
Next, we insert the divided differences into formula~\eqref{e:Newton polynomial} and
obtain the approximate Newton interpolating polynomial $p$. Finally, we substitute the
matrix $A$ into the polynomial $p$ and obtain $p(A)$ which is approximately equals
$f(A)$. We note that the employment of the Schur form may simplify this substitution.

We note that in exact arithmetic the offered algorithm results in the same approximation
of $f(A)$ as the algorithm described in~\cite{Davies-Higham2003}. Indeed, it is easy to
see that the calculation of $p(A)$, where $p$ coincides with $h_j$ in a neighbourhood of
$S_j$, according to the algorithm from~\cite{Parlett74}, gives just the result
from~\cite{Davies-Higham2003}. Nevertheless the described algorithm needs not a
reordering of eigenvalues in the Schur triangular representation and solving the
Sylvester equations.

Let us discuss briefly how to choose $\delta$. According to the definition, $\delta$ is
the estimate of the distance between the clusters from below, i.e., $\delta$ is the best
(known) constant in the estimate $\mu_i-\nu_j\ge\delta$, where $\mu_i$ and $\nu_j$ are
from different clusters. Hence the calculation of $\mu_i-\nu_j$ may result in the drop of
approximately $\log_{10}|\mu_i|-\log_{10}{\delta}$ significant decimal digits in floating
point arithmetic. So, if the desirable final accuracy and the accuracy of the eigenvalues
$\mu_i$ are known, one can estimate the smallest admissible $\delta$.

A visual control may be very useful. Since we assume that the eigenvalues $\mu_i$ are
known, we may display the spectrum of the matrix $A$. The figure can help to choose
clusters $S_j$ and their centres $\bar\mu_j$ almost manually. If the eigenvalues have a
kind of uniform distribution, the algorithm can be applied only in two extreme ways (it
becomes trivial, but not obligatorily useless): either we interpret the whole spectrum as
the only cluster (in this case the application of the algorithm implies the replacement
of $f$ by its Taylor polynomial) or we consider each eigenvalue as a separate cluster (in
this case the application of the algorithm is equivalent to the calculation of $p(A)$,
where $p$ is the Newton interpolating polynomial~\eqref{e:Newton polynomial} with divided
differences calculated by direct formulas~\eqref{e:divided differences}).

\section{Numerical experiments}\label{s:num exper}
In all numerical experiments we compute approximately $e^A$. The minimal possible
distance $\delta=0.01$ between clusters and the maximal possible size $\eta=0.001$ of the
clusters are the same for all experiments.

The results of the experiments are presented in Table~1 (numbers smaller than $10^{-10}$
are replaced with zeros). Each row of Table~1 describes the joint result of 1000
experiments with the same parameters. The headings of Table~1 have the following
meanings: $n$ is the order of the matrix $A$; $K$ is the maximal admissible number $k_j$
of eigenvalues in one cluster; $\gamma$ is the parameter from formula~\eqref{e:h} (it is
assumed that $\gamma$ is the same for all clusters $S_j$); $E(\varkappa(T))$ is the
sample mean of the condition number $\varkappa(T)=\Vert T\Vert\cdot\Vert T^{-1}\Vert$ of
the similarity transformation $T$ defined below; $E\bigl(\frac{\Vert p(A)-e^A
\Vert}{\Vert e^A\Vert}\bigr)$ is the sample mean of the relative accuracy $\frac{\Vert
p(A)-e^A \Vert}{\Vert e^A\Vert}$, where $e^A$ is the exact exponential of $A$ and $p(A)$
is computed according to the algorithm from Section~\ref{s:algorithm};
$\text{Max}(\varkappa(T))$ and $\text{Max}\bigl(\frac{\Vert p(A)-e^A \Vert}{\Vert
e^A\Vert}\bigr)$ are maximum values of the same quantities; $M$ is the number of
experiments (from 1000) that resulted in $\frac{\Vert p(A)-e^A \Vert}{\Vert
e^A\Vert}>0.001$. By $\Vert A\Vert$ we mean the operator norm
\begin{equation*}
\Vert A\Vert=\max\{\,\Vert Ax\Vert_2:\,\Vert x\Vert_2=1\,\},
\end{equation*}
where $\Vert x\Vert_2=\sqrt{|x_1|^2+\dots+|x_n|^2}$.

Each numerical experiment consists in the following. First, the sequence $k_1$, $k_2$,
\dots, $k_\beta$ of multiplicities is constructed. The numbers $k_j$ are defined as
random whole numbers from $[1,K]$; the last number $k_\beta$ is chosen so that
$k_1+\dots+k_\beta=n$. After that the approximate centres $\bar\mu^{(\text{ini})}_j$,
$j=1,\dots,\beta$, are defined as random numbers from $[-2,0]\times[-i\pi,i\pi]$. If
$\min_{i\neq j}|\bar\mu^{(\text{ini})}_i-\bar\mu^{(\text{ini})}_j|<\delta$, then the
sequence $\mu_1$, $\mu_2$, \dots, $\mu_\beta$ is rejected and another sequence is chosen.
The eigenvalues $\mu_i$, $i=1,\dots,n$, of $A$ from the $j$th cluster are defined as
$\bar\mu^{(\text{ini})}_j$ plus random numbers from $[-\eta,\eta]\times[-i\eta,i\eta]$
(we recall that the cluster $S_j$ contains $k_j$ eigenvalues). Let $\Lambda$ be a
diagonal matrix with the diagonal elements $\mu_i$, $i=1,\dots,n$, and $T$ be a matrix
(similarity transformation) consisting of random numbers from $[-1,1]\times[-i,i]$. We
set $A=T^{-1}\Lambda T$ (we never met a case when $T$ is not invertible). We take for the
exact matrix $e^A$ the matrix $T^{-1}e^{\Lambda}T$. Finally we compute the approximation
$p(A)$ of $e^A$ according to the algorithm from Section~\ref{s:algorithm} and compare
$p(A)$ with the exact matrix $e^A$.

The first three rows of Table 1 corresponds to the case $K=1$, which means that there are
no close eigenvalues; in this case the  method under discussion coincides with the
ordinary usage of the Newton interpolating polynomial (thus, the value of $\gamma$ makes
no difference). The numerical experiments show that the direct usage of the Newton
interpolating polynomial (within the framework of this experiment, i.e., for
$\delta=0.01$, the function $f(\lambda)=e^\lambda$ etc.) is not reliable if $n>60$. The
next four rows show that in the case $n=40$ the algorithm should be applied with care.
The case where $K\le30$ can be considered as more or less admissible. The case $K\le20$
is quite reliable; in this case taking $\gamma>-1$ is not very essential.

\begin{table}[th]\caption{Results of numerical experiments}\label{tab}
\begin{tabular}{|c|c|c||c|c||c|c|c|}\hline
  $n$ & $K$ & $\gamma$ & $\text{Max}(\varkappa(T))$ &
$\text{E}(\varkappa(T))$ & $\text{Max}\bigl(\frac{\Vert p(A)-e^A \Vert}{\Vert
e^A\Vert}\bigr)$ & $\text{E}\bigl(\frac{\Vert p(A)-e^A \Vert}{\Vert e^A\Vert}\bigr)$ &
$M$\\\hline\hline
 70 & 1 & $-1$ & 211500 & 4664 & 158.6 & 0.2101 & 82\\\hline
 60 & 1 & $-1$ & 107400 & 3196 & 0.001646 & $5.03\times 10^{-6}$ & 2\\\hline
 50 & 1 & $-1$ & 166600 & 2375 & $5.839\times 10^{-6}$ & $8.722\times 10^{-9}$ & 0\\\hline
\hline
 40 & 2 & $-1$ & 46850 & 1648 & 32.39 & 0.05746 & 56\\\hline
 40 & 2 & 5 & 70450 & 1749 & 0.001696 & $1.699\times 10^{-6}$ & 1\\\hline
 40 & 4 & $-1$ & 138000 & 1810 & $7.793\times 10^7$ & 109000 & 146\\\hline
 40 & 4 & 5 & 63870 & 1817 & $7.759\times 10^8$ & 789400 & 67\\\hline
 \hline
 30 & 2 & -1 & 86900 & 1062 & 356.5 & 0.3565 & 1\\\hline
 30 & 2 & 5 & 70990 & 1043 & 0 & 0 & 0\\\hline
 30 & 4 & $-1$ & 34700 & 1015 & 16.85 & 0.01878 & 17\\\hline
 30 & 4 & 5 & 42530 & 1011 & 0.000098 & $1.037\times 10^{-7}$ & 0\\\hline
 30 & 8 & $-1$ & 49010 & 1099 & 861200 & 1144 & 36\\\hline
 30 & 8 & 5 & 57790 & 1163 & $1.88\times 10^8$ & 190400 & 13\\\hline
 \hline
 20 & 4 & $-1$ & 97310 & 668 & $5.046\times 10^{-7}$ & $2.032\times 10^{-9}$ & 0\\\hline
 20 & 4 & 5 & 15090 & 582.3 & 0 & 0 & 0 \\\hline
 20 & 8 & $-1$ & 35840 & 589.2 & 3.045 & 0.003217 & 3\\\hline
 20 & 8 & 5 & 24080 & 551.5 & $6.936\times 10^{-6}$ & $1.173\times 10^{-8}$ & 0\\\hline
 20 & 16 & $-1$ &  27380 & 588.3 & 1155. & 1.155 & 2\\\hline
 20 & 16 & 5 & 21470 & 615.3 & 0.0009528 & $9.528\times 10^{-7}$ & 0\\\hline
\end{tabular}
\end{table}

\section{Symbolic calculation of the impulse response}\label{s:pulse char}
The offered algorithm can be applied to a calculation of analytic functions $f_t$ of $A$
depending on a parameter, e.g., $f_t(\lambda)=e^{\lambda t}$. In this Section we discuss
an example of such a problem.

Let us consider the dynamical system
 \begin{equation}\label{e:dynamic system}
 \begin{split}
x'(t)&=Ax(t)+bu(t),\\
y(t)&=\langle x(t),d\rangle
 \end{split}
 \end{equation}
with the scalar input $u$ and the scalar output $y$. Here, $b,d\in\mathbb C^n$ are given
vectors, and the symbol $\langle\cdot,\cdot\rangle$ means the inner product. In the
majority of applications the spectrum of $A$ is contained in the open left half plane. We
also note that in many applications (e.g., in control problems) a high accuracy (more
than 0.001) of the solution makes no sense, because the accuracy of the initial physical
model is essentially lower.

\emph{The impulse response} of system~\eqref{e:dynamic system} is the solution $y$
of~\eqref{e:dynamic system} that corresponds to the input $u(t)=\delta(t)$, where
$\delta$ is the Dirac function, and equals zero when $t<0$. It is well known that the
impulse response of system~\eqref{e:dynamic system} can be represented in the form
\begin{equation}\label{e:pulse char}
t\mapsto\langle d,e^{At}b\rangle,\qquad t>0.
\end{equation}

We assume that the order $n$ of the matrix $A$ is rather small ($n$ is about 10). In this
case the algorithm under discussion allows one to present the impulse response in a
symbolic form, i.e., in the form of a formula. If the order $n$ of $A$ is large, the
symbolic representation of the impulse response is too cumbersome; in this case, some
reduced order method~\cite{Antoulas,Benner-Mehrmann-Sorensen} can be applied in advance.

We interpret the matrix $e^{At}$ as the analytic function $f_t(\lambda)=e^{\lambda t}$
depending on the parameter $t$ applied to the matrix $A$. Hence we are able to make use
of the algorithm from Section~\ref{s:algorithm}. In a neighbourhood of the cluster $S_j$
we define $h_j$ as the Taylor polynomial of the function $f_t(\lambda)=e^{\lambda t}$
depending on the parameter $t$:
\begin{equation*}
h_j(\lambda)=\sum_{\alpha=0}^{k_j+\gamma_j}\frac{t^\alpha e^{\bar\mu_jt}}{\alpha
!}\,(\lambda-\bar\mu_j)^\alpha.
\end{equation*}
This is a special case of formula~\eqref{e:h_j} with the coefficients
\begin{equation*}
c_{\alpha j}(t)=\frac{t^\alpha e^{\bar\mu_jt}}{\alpha!},\qquad j=1,\dots,\beta,\quad
\alpha=0,\dots k_j+\gamma_j,
\end{equation*}
depending on the parameter $t$. Therefore proposition~\ref{p:Newton projector2} remains
valid. As a result the divided differences $\Delta_i^{i+m}$, the coefficients of the
Newton interpolating polynomial $p$ from Section~\ref{s:algorithm}, and the elements of
the matrix $p(A)$ are linear combinations of the functions $c_{\alpha j}$. Thus finally
we arrive at a formula similar to the classical representation of the impulse response in
the form of a linear combination of the functions $t\mapsto\frac{t^\alpha
e^{\mu_it}}{\alpha!}$.

But there are some distinctions. We recall that the degree $k_j+\gamma_j$ of the
polynomial $h_j$ may happen to be greater than the number $k_j$ of points in the cluster
$S_j$ diminished by 1 (provided that $\gamma_j>-1$). In this case the power $\alpha$ in
the expression $t^\alpha e^{\bar\mu_jt}$ may turn out to be more than $k_j-1$, which is
unusual for the exact representation of the impulse response. Besides, the numbers
$\bar\mu_j$ may not coincide precisely with the eigenvalues $\mu_i$ of $A$.

Numerical experiments show that the substitution $t=1$ into the matrix-function $t\mapsto
e^{At}$ computed in the described way gives the result which coincides within the
accuracy of calculations ($10^{-16}$) with $e^A$ computed in accordance with the
algorithm from Section~\ref{s:algorithm}.

We conclude with the well known remark. Representation~\eqref{e:pulse char} of the
impulse response shows that when we substitute the powers $A^i$ of the matrix $A$ into
the polynomial $p$ it is enough to restrict ourselves to the calculation of $A^ib$
instead of the whole $A^i$.

\section{Conclusion}
An algorithm for computing an analytic function of an $n\times n$-matrix $A$ is
presented. This algorithm is a modification of the algorithms from
\cite{Parlett74,Parlett76,Parlett-Ng,
Kagstrom77,McCurdy-Ng-Parlett,Davies-Higham2003,Higham-Al-Mohy2010}. The algorithm works
correctly when $A$ is allowed to have close eigenvalues. It reliably works when $n\le20$.
The algorithm can be used for the calculation of a symbolic representation of the impulse
response of a dynamical system of small dimension.

\end{document}